\documentclass{article}

\usepackage{arxiv}

\usepackage[utf8]{inputenc} 
\usepackage[T1]{fontenc}    
\usepackage{hyperref}       
\usepackage{url}            
\usepackage{booktabs}       
\usepackage{amsfonts}       
\usepackage{nicefrac}       
\usepackage{microtype}      
\usepackage{lipsum}
\usepackage{amsthm}
\usepackage{graphicx}
\usepackage{xcolor}
\usepackage{amsmath}
\usepackage{amssymb}
\usepackage{bookmark}
\usepackage{array}
\usepackage{algorithm}
\usepackage[noend]{algpseudocode}
\usepackage{subfigure}

\newtheorem{theorem}{Theorem}

\newtheorem{definition}[theorem]{Definition}
\newtheorem{remark}[theorem]{Remark}

\newcommand{\dd}{\mathrm{d}}
\newcommand{\J}{\mathcal{J}}

\DeclareMathOperator*{\argmin}{\arg\!\min}

\title{Analysis of Seismic Inversion with Optimal Transportation and Softplus Encoding}

\author{
  Lingyun Qiu \\
  Yau Mathematical Sciences Center\\
  Tsinghua University \\
  Beijing China 100084 \\
  \texttt{lyqiu@tsinghua.edu.cn} \\  
}

\begin{document}
\maketitle

\begin{abstract}
  This paper is devoted to theoretical and numerical investigation of the local minimum issue in seismic full waveform inversion (FWI).
  This paper provides a mathematical analysis of optimal transportation (OT) type objective function's differentiability and proves that the gradient obtained in the adjoint-state method does not depend on the particular choice of the Kantorovich potentials. A novel approach using the softplus encoding method is presented to generalize and impose the OT metric on FWI. This approach improves the convexity of the objective function and mitigates the cycle-skipping problem. The effectiveness of the proposed method is demonstrated numerically on an inversion task with the benchmark Marmousi model.
\end{abstract}


\section{Introduction}
Seismic full waveform inversion uses kinematic and dynamical information of the seismic wavefield to build the subsurface velocity model, which accurately depicts the geological structures.
 Mathematically, FWI is formulated as a nonlinear inverse problem matching modeled data to the recorded field data \cite{Tarantola1984}. 
 It can be solved as a PDE-constrained optimization problem, and a least-squares objective function is used for measuring the data misfit. 
 The objective function is minimized with respect to the model parameter, and the model update is computed using the adjoint state method \cite{Plessix2006}. FWI can produce high-resolution models of the subsurface compared to ray-based methods. However, FWI is often an ill-posed problem due to the band-limited nature of the seismic data and the limitations of the acquisition geometries. 
 
 The least-square formulation of FWI, when the initial model is far from the true model and the seismic data lack of low-frequency information, tends to produce many local minima. That is the so-called cycle-skipping issue. 
 The cause of this issue is that only the pointwise amplitude difference is measured with the $L_2$ norm, while the phase or travel-time information embedded in the data is more critical for the inversion. There are different approaches proposed to capture more accurate kinematic information, such as dynamic time warping and convolution-based methods. This information is then used for the optimization to convexify the objective function or enlarge the true solution valley. In this direction, we mention the works in \cite{Luo1991,VanLeeuwen2010,Luo2011,Ma2013,Warner2014}.
 An alternative approach to reshape the objective function is to extend the parameter space \cite{Symes2015,Symes2017} or use an auxiliary wavefield \cite{VanLeeuwen2013,Leeuwen2015,Wang2016} in a non-physical way so that the data can be easily fitted. Then, one can get the physical model back by using an annihilator or gradually tightening the PDE-constraints.
 
 Another approach involves the use of Wasserstein metrics. The Wasserstein distance and OT theory were first brought up to seek the optimal cost of rearranging one density into the other, where the transportation cost per unit mass is the Euclidean distance or Manhattan distance. It can be traced back to the mass transport problem proposed by Monge in 1780s and its relaxed formulation by Kantorovich in the 1940s. Since then, it has become a classical subject in probability theory, economics, computer vision, optimization, and partial differential equations. Recently, the Wasserstein distance and its variants are proposed to replace the distance for the objective function in FWI \cite{Engquist2014,Metivier2016,Qiu2017,Engquist2019,Metivier2019,Sun2019,Li2020,Engquist2020}. Its successful applications are partly due to its Lagrangian nature to capture the important variations of signals, such as translation(time-shift) and dilation.

Among various strategies for mitigating the cycle-skipping issue in seismic inversion, using OT-based objective functions has been demonstrated to be one of the most effective approaches. However, these approaches have three points that require further investigation. First of all, there is sufficient evidence that FWI has a strong path dependence. However, the Kantorovich potential generally does not have uniqueness, hence the associated model gradient obtained using the adjoint-state method is nor unique. Secondly, the inversion's effectiveness depends crucially on the appropriate underlying encoding method to transfer the seismogram to probability density functions(PDFs). Last but not least, while using OT metrics is mainly concerned with the cycle-skipping problem, it tends to give smooth inversion results and lacks high-resolution details. In order to delineate usage scenarios, frequency sensitivity analysis is crucial. 

Our goal of the present paper is to provide a rigorous description of the gradient-based methods and proper encoding method for the seismic inverse problem using OT. These findings are essential for a rigorous interpretation of the numerical observations. Several objectives are pursued in this paper. First, a rigorous proof is presented on the directional differentiability of the transportation cost as a function in $L_2$.  We also perform a frequency sensitivity analysis of the OT objective function using the Fourier series. Second, an encoding method using the softplus function is introduced, and then it is proved that the gradient obtained using the adjoint state method is well-defined and unique. Finally, by applying it to a simple convexity test and an inverse problem on the benchmark Marmousi model, the feasibility of the proposed method is demonstrated.
  
The paper is organized as follows. The necessary notations and properties of the quadratic Wasserstein distance, especially the efficient solution in the unidimensional case, are discussed in Section~\ref{sec:W2}. As shown in Section~\ref{sec:FWI}, one only needs to change the adjoint source when switching from $L_2$ to other metrics in the objective function, and an adjoint source involving the first Kantorovich potential and the gradient of the encoding map is used for the OT one. In section~\ref{sec:analysis}, we investigate the frequency sensitivity of OT and specify a low-frequency enhancement of it; a rigorous proof of the directional differentiability and uniqueness of the gradient is also presented. The desired properties of encoding methods and an effective approach using the softplus function are illustrated in Section~\ref{sec:encode}, and two numerical examples are shown in Section~\ref{sec:num}. For completeness and reproducibility of the results,  pseudo-code of the proposed method is presented in an appendix.

\section{The quadratic Wasserstein distance}
\label{sec:W2}
This section introduces the quadratic Wasserstein distance used to measure the difference between data. We begin with some standard notations and necessary properties.

\subsection{Notation}
Throughout this paper, we shall consider probability measures that are absolutely continuous with respect to the Lebesgue measure and with a finite moment of order $2$ on a simply connected and compact domain in $\mathbb{R}^n$. Hence, we identify the induced measure with its Radon-Nikodym derivative with respect to the Lebesgue measure and write $\dd \mu(x) = \mu(x) \dd x$. 
The measure and its Radon-Nikodym derivative will not be distinguished, as it should be clear from the context.
All the measures considered here are built from the solution to wave equations. The regularity condition is clearly satisfied, and the limited-time/space measurement of the data leads to the boundedness of the domain.
When no ambiguity arises, we denote for brevity by $\mathcal{P}$ the set of all absolutely continuous measures with a finite moment of order $2$ on the given domain. Whenever $T$ is a map from a measure space $X$, equipped with a measure $\mu$, to an arbitrary space Y, we denote by $T\#\mu$ the push-forward of $\mu$ by $T$.
Explicitly,
$$
T\#\mu(A) = \mu(T^{-1}(A)), \quad  \mbox{with } T^{-1}(A) = \{x\in X \mid T(x) \in A).
$$
For non-negative functions $f$ and $g$, we write $T\#f=g$ to mean that the measure having Lebesgue density $f$ is pushed forward to the measure having Lebesgue density $g$.

Afterward, three forms for OT are introduced: Monge's problem, Kantorovich's formulation, and its dual form. Actually, in this paper's setting, there exists a unique solution to Kantorovich's problem, which is also a unique solution to Monge's problem. Even so, from a numerical point of view, it is beneficial to switch among these formulations.

\begin{definition}[Monge's OT problem]\label{def:Monge}
Let $\mu, \nu \in \mathcal{P}$. Minimize
\begin{equation}
  I[T] = \int |x - T(x)|^2 \, \dd \mu(x)
\end{equation}
over the set of all measurable maps $T$ such that $T\# \mu = \nu$.
\end{definition}

\begin{definition}[Kantorovich's OT problem]
Let $\mu, \nu \in \mathcal{P}$. Minimize
\begin{equation}\label{K-problem}
  I[\gamma] = \int \frac{1}{2}|x - y|^2 \, \dd \gamma(x,y)
\end{equation}
over the set of all coupling measures, which admit $\mu$ and $\nu$ as marginals on the first and second factors respectively, i.e.,
\begin{equation}
  \int \left(\varphi(x) + \psi(y) \right) \dd \gamma(x,y) = \int \varphi \dd \mu + \int \psi \dd \nu ,
\end{equation}
for all measurable functions $\varphi \in L^1(\dd \mu)$ and $\psi \in L^1(\dd \nu)$.
\end{definition}
Monge's formulation leads to a highly nonlinear minimization problem with nonlinear constraints, while the Kantorovich's formulation revisits the same problem from a linear programming point of view. It leads to a linear minimization under convex constraints. Kantorovich 's problem, thus, admits a duality interpretation, which turns out to be a powerful tool in the OT theory.

\begin{theorem}[Kantorovich duality]\label{thm:dual}
The minimum of Kantorovich's problem (\ref{K-problem}) is equal to the supremum of
\begin{equation}\label{dual-form}
  \int \varphi \, \dd \mu + \int \psi  \, \dd \nu
\end{equation}
over all pairs  $(\varphi, \psi) \in L^1(\dd \mu) \times L^1(\dd \nu)$ such that $\varphi(x) + \psi(y) \leq \frac{1}{2}|x-y|^2$.
\end{theorem}

In the setting of this article, Monge's problem and Kantorovich's formulation both have solutions and reach the same optimal value. That is, the supremum in Theorem~\ref{thm:dual} and the infimum in Definition~\ref{def:Monge} are equal and both attainable. Moreover, for quadratic cost $c(x,y)=\frac{1}{2}|x-y|^2$, a transference plan is optimal if and only if it is concentrated on the sub-differential of a convex function and such a plan is $\dd \mu$-unique \cite{Villani2008}, and the relationship between the maximizing pair $(\phi, \psi)$ and the optimal transference plan $T$ is
\begin{equation}\label{plan-dual-relation}
       \psi(y) =  \min_x \frac{1}{2} |x-y|^2 -\varphi(x), \quad
     T(x) = \nabla \left(\frac{|x|^2}{2} -\varphi(x)\right).
\end{equation}
This formulation, especially the second equality, turns out to be a handy tool in the calculation of the first variation from the numerical point of view.
The Wasserstein distance between $\mu$ and $\nu$ is defined as
\begin{equation}\label{def:W2}
  W_2(\mu, \nu) = \min_T I[T]^{1/2} = \min_\gamma I[\gamma]^{1/2}.
\end{equation}
For simplicity, we consider the second power of $W_2$, which is the optimal total transportation cost $\mathcal{T}(\mu, \nu) = W_2^2(\mu, \nu)$.

\subsection{OT on the real line}
The one-dimensional case is of particular interest, as its equivalent definition does not involve solving a minimization problem. It can be solved explicitly and efficiently with a linear computational complexity from a computational point of view. From a theoretical point of view, the $1D$ Wasserstein distance, as a function of its first argument, is strongly convex along the geodesic. All high-dimensional ones are not even convex along the geodesic(see e.g. \cite[Example $9.1.5$]{Ambrosio2008}). From the perspective of the seismic inverse problem, this leads us to consider using the OT metric on the time variable combined with the least-squares on the spatial variable, rather than the high-dimensional Wasserstein distances on space-time variable.

The following theorem states a solution to the Monge-Kantorovich problem on the real line in terms of cumulative distribution functions.

\begin{theorem}[OT theorem on $\mathbb{R}$ \cite{Villani2003,Santambrogio2015}]
\label{thm:w2-1d}
Let $p_0, p_1\in \mathcal{P}(\mathbb{R})$ be two probability measures on the real line. $f_0$ and $f_1$ are their cumulative distribution functions:
\[
f_k(x) = \int_{-\infty}^x \, \dd p_k, \quad k = 0,1.
\]
The pseudo-inverse of a non-decreasing and right-continuous function $f$ is defined by
\[
f^{[-1]}(x) = \inf \{t\in \mathbb{R} \mid f(t) > x \}.
\]

Then, there exists a unique non-decreasing map $T:\mathbb{R}\rightarrow \mathbb{R}$ given by $T(x) = f_0^{[-1]}(f_1(x))$ such that $p_0(T(x)) =p_1(x)$. The map $T$ is optimal in the Monge-Kantorovich problem for the quadratic cost function. Moreover, the value of the optimal transport cost is
\begin{eqnarray}
  W_2(p_0,p_1) & = \left( \int_0^1 (f_0^{[-1]}(s) - f_1^{[-1]}(s))^2 \dd s \right)^{1/2} \nonumber
  \\
   & = \left( \int_{-\infty}^{+\infty} (f_0^{[-1]}(f_1(t)) - t)^2 p_1(t) \dd t  \right)^{1/2}.
\end{eqnarray}
\end{theorem}

\begin{remark}
There are several aspects to be mentioned here regarding the optimal transport map, $T$. First, if the two measures are atomless and strictly positive, and hence, the cumulative distribution functions are continuous and strictly monotone, then one would have
\[
T= f_0^{-1} \circ f_1.
\]
Second, from the explicit form of $T$, we conclude that the regularity of $T$ is one degree higher than that of the measures. Higher regularity leads to a smoother effect. This has been observed in the numerical experiments.
Third, the form of $T$ implies that the solution to the transportation problem is given by a monotone rearrangement of $p_0$ onto $p_1$. This leads to the algorithm with computational cost $\mathrm{O}(N)$ for computing the transportation cost and its first variation. Please refer to the Appendix for more details.
Fourth, the transportation map $T$ is optimal not only for the quadratic cost, but also for all cost functions in the form of $c(x,y) = h(y-x)$ with $h$ being a convex function. In particular, the optimal transportation cost associated with the cost function $c(x,y) = |x-y|$ is
\[
W_1(p_0,p_1) = \int_0^1 \left|f_0^{[-1]}(s) - f_1^{[-1]}(s)\right|\, \dd s = \int_{-\infty}^{+\infty} \left|f_0(x) - f_1(x)\right|\, \dd x.
\]
Finally, the first variation of the transportation cost is given by
\begin{multline}\label{tran-deri}
    \frac{\partial W_2^2(p_0,p_1)}{\partial p_0} = \\
 (f_1^{[-1]}(f_0(t)) -t)^2
+ \int_t^1 2 \left.\frac{\partial f_1^{[-1]}(x)}{\partial x} \right|_{x=f_0(s)} \left(f_1^{[-1]}(f_0(s)) -s\right) p_0(s) \, \dd s.
\end{multline}
To simplify the calculation and avoid the differentiation, which may cause some numerical error, the second term in the above formula can be rewritten as
\begin{equation}
  2 \int_{f_1^{[-1]}(f_0(t))}^{1} \left(s - f_0^{[-1]}(f_1(s))\right) \, \dd s.
\end{equation}
Here the inverse function theorem is applied.

\end{remark}

\section{Full waveform inversion}
\label{sec:FWI}
In this section, we first briefly review the theory of FWI and the adjoint state method. Then, an analog of the adjoint wavefield using transportation distance is developed. The differentiability and uniqueness will be analyzed in subsequent sections.
In Section~\ref{sec:num}, it will be used in conjunction with the softplus encoding method to perform numerical experiments.

We start with the acoustic wave equation in the time domain governed by
\begin{equation}\label{eq:wave}
  \left( m(x) \frac{\partial^2}{\partial t^2} - \nabla \cdot\left(\frac{1}{\rho(x)}\nabla\right) \right) u(x,t) = f(x,t),
\end{equation}
where $m$ is the reciprocal of the bulk modulus, $\rho$ is the density, and $u$ and $f$ stand for the pressure wavefield and source term, respectively. We symbolize the relationship between the model parameters and the observed wavefield by an operator $F$, which is also referred to as the forward operator,
\begin{equation}\label{eq:fwd}
F(m,\rho,f) = u|_\Gamma.
\end{equation}
$\Gamma$ stands for the receiver geometry, which is usually a portion of a surface or a collection of discrete points.

The goal of the inverse problem is to reconstruct the model parameters from the measured data. Usually, the inverse problem is posed as a nonlinear least-squares optimization problem,
\begin{equation}\label{eq:optim}
  \min_{m, \rho, f} \J(m) = \frac{1}{2}\|F(m,\rho,f) - d\|_2^2,
\end{equation}
where $\J$ is the misfit function and $\|\cdot\|_2$ is the $L_2$ norm. That is, to choose the model parameters such that the correspondingly simulated waveform yields the minimum difference away from the measured data in the $L_2$ sense.
For simplicity, we assume the density $\rho$ and source term $f$ are known in this work. Hence we omit the explicit dependence on $\rho, f$ in (\ref{eq:fwd}) and (\ref{eq:optim}) in the following sections.

\subsection{Adjoint state method}
Modern techniques for the seismic inverse problems involve the use of data with size ranging from gigabytes to terabytes or even petabytes.
The adjoint state method plays a significant role in the computational aspect of large scale optimization problems. For completeness, a simple description is included here in a general setting. For more details on this topic, please refer to \cite{Plessix2006}.

Suppose the misfit function is $\J(u(m))$, where $u$ and $m$ stand for the state variable and model parameter, respectively. $u$ and $m$ satisfy the state equation $\Phi(m,u(m)) = 0$. For the gradient-based method, the total derivative ${\delta \J}/{\delta m}$ needs to be computed to assess the sensitivity of the misfit function to the model parameter.
The gradient ${\delta \J}/{\delta m}$ is simply
\begin{equation}
  \frac{\delta \J}{\delta m} = \left\langle \frac{\delta \J}{\delta u} \, , \frac{\delta u}{\delta m} \right\rangle ,
\end{equation}
where the inner product acts in the space of $u$, and ${\delta u}/{\delta m}$ is a linear operator acting on perturbations on $m$ and returning perturbations on $u$.
In the context of FWI, the difficulty of numerically evaluating ${\delta \J}/{\delta m}$  lies in the evaluation of the wavefield perturbation $\delta u$ for all possible model perturbation $\delta m$.
The adjoint state method answers the question, ``How to efficiently calculate ${\delta \J}/{\delta m}$ without evaluating ${\delta u}/{ \delta m}$ explicitly?"
 
Let us define the adjoint state variable $v$ as the solution of the adjoint state equation,
\begin{equation}
\label{eq:adjoint-abstract}
\left(\frac{\partial \Phi}{\partial u} \right)^* v = \frac{\delta \J}{\delta u}.
\end{equation}
From the state equation, we know that 
\begin{equation}
\label{eq:state-eq-rela}
\frac{\partial \Phi}{\partial u} \frac{\delta u}{\delta m} + \frac{\partial \Phi}{\partial m} =0.
\end{equation}
It follows that
\begin{eqnarray}
\frac{\delta \J}{\delta m} & = & \left\langle \frac{\delta \J}{\delta u} \, , \frac{\delta u}{\delta m} \right\rangle \nonumber \\
& = & \left\langle \left(\frac{\partial \Phi}{\partial u}\right)^* v \, , \frac{\delta u}{\delta m} \right\rangle  \\
& = & \left\langle v \, , \frac{\partial \Phi}{\partial u} \, \frac{\delta u}{\delta m} \right
\rangle \nonumber \\
& = & \left\langle v \, , -\frac{\partial \Phi}{\partial m} \right \rangle . \nonumber
\end{eqnarray}
In the above identities, we omit the explicit dependence of the inner products on the associated spaces for simplicity. Indeed, from this formulation, one observes that as long as the wavefield $u$ is still the intermediate in the construction of the data misfit function $\J$, only the adjoint state variable $v$ depends on the specific form of $\J$. Furthermore, only the adjoint source term $ \frac{\delta \J}{\delta u}$ needs to be modified for different misfit functions as long as it is of the form $\J=\J(u(m))$.

In the conventional FWI with least-square misfit function
\[
\J = \frac{1}{2}\|u - d\|_2^2, \quad u = F(m),
\]
we have that
\[
\frac{\delta \J}{\delta u} = u - d.
\]
Applying the adjoint state method gives
\begin{equation}
\label{eq:L2-adj-stat}
\frac{\delta \J}{\delta m} = \left\langle v \, , -\frac{\partial \Phi}{\partial m} \right \rangle ,
\end{equation}
where the adjoint state variable $v$ solves the adjoint state equation
\begin{equation}
\left\{
\begin{array}{>{\displaystyle}r>{\displaystyle}l}
\left( m(x) \frac{\partial_2}{\partial t^2} - \nabla \cdot\left(\frac{1}{\rho(x)}\nabla\right) \right) v(x,t) &  =  u - d , \\
 v(x,T)  & =  0, \\
 \partial_t v(x,T) & =  0 .
 \end{array}
\right.
\end{equation}

For the FWI with quadratic Wasserstein norm and proper encoding, the data misfit function is defined as
\begin{equation}
\label{obj-W2}
\J = W_2^2(\tilde{u},\tilde{d}), \quad u = F(m), \, \tilde{u} = \mathcal{D}(u), \, \tilde{d} = \mathcal{D}(d),
\end{equation}
where $\mathcal{D}$ is the encoding operation from seismic data to equal-mass non-negative measures.
It follows that
\begin{equation}
\frac{\delta \J}{\delta u} = \left\langle \frac{\dd W_2^2(\tilde{u},\tilde{d}) }{\dd \tilde{u}}  ,\, \frac{\dd \mathcal{D}(u)}{\dd u} \right \rangle = \mathcal{D}'[u]^* (\varphi),
\end{equation}
where $\varphi$ is the Kantorovich potential of $W_2^2(\tilde{u},\tilde{d})$ associated with $\tilde{u}$.
Then, applying the adjoint state method, we obtain that
\begin{equation}
\label{eq:W2-adj-stat}
\frac{\dd \J}{\dd m} = \left\langle v \, , -\frac{\partial \Phi}{\partial m} \right \rangle ,
\end{equation}
where the adjoint state variable $v$ solves the adjoint state equation
\begin{equation}
\label{eq:adjoint-eq-W2}
\left\{
\begin{array}{>{\displaystyle}r>{\displaystyle}l}
\left( m(x) \frac{\partial^2}{\partial t^2} - \nabla \cdot\left(\frac{1}{\rho(x)}\nabla\right) \right) v(x,t) &  = \mathcal{D}'[u]^* (\varphi) , \\
 v(x,T)  & =  0, \\
 \partial_t v(x,T) & =  0 .
 \end{array}
\right.
\end{equation}

\section{Wasserstein metric from a seismic inverse problem perspective}
\label{sec:analysis}
In this section, we discuss the features of the quadratic Wasserstein metric from a seismic inverse problem perspective. We start by investigating the frequency sensitivity of $\mathcal{T}$. It is proved that $\mathcal{T}$ emphasizes the low-frequency components not only locally in the linearization regime but also a global sense. 
This also reveals that the sensitivity of the solution is small in highly oscillating data.
Next, we present the rigorous definition of a set, says $\mathcal{D}$, in which the optimization is performed. We show the Euclidean differentiability of the transportation cost, and that the gradient is unique up to an additive constant for any element in $\mathcal{D}$. This set will be used as a desirable image domain to design the encoding mapping. 

\subsection{Frequency sensitivity of \texorpdfstring{$W_2$}{W2}}
A long-standing view in seismic inversion starts with low-frequency data, which contain large-scale, kinematically relevant components of the velocity model. The low-to-high frequency-continuation schemes \cite{Bao2015,Hoop2015,Bunks1995a,Virieux2009,Fichtner2013} help FWI mitigate the cycle-skipping issue, i.e., the local minimum problem. At the same time, an overly detailed frequency division will slow down the entire inversion process significantly. As is well known, the quadratic Wasserstein distance $W_2(\mu,\cdot)$ is asymptotically equivalent to a weighted $\dot{H}^{-1}(\dd \mu)$, where $\dot{H}^{-1}$ denotes the dual space of the space of zero-mean $H^1$ function. It is also well known that $L_2$ measures different frequency components equally, and $\dot{H}^{-1}$ attenuates them with a polynomial weight of order $|k|^{-1}$. 
The following theorem shows a non-asymptotically similar behavior of $W_2$ and $\dot{H}^{-1}$.
\begin{theorem}\label{thm:freq-sens}
  Assume that $\mu_0, \mu_1 \in \mathcal{P}(S^1)$, where $S^1$ stands for the unit circle, and 
  \begin{equation}
     \mu_1 = \mu_0  + \sum_{k\in \mathbb{Z}^+} \left(a_k \cos{(k\theta)} + b_k \sin{(k\theta)}\right) .
  \end{equation}
  Note that the $0$-frequency amplitude vanishes since $\int \dd \mu_0 = \int \dd \mu_1$.
  If 
  \begin{equation}
    \nu = \mu_0  - \sum_{k\in \mathbb{Z}^+} \left((a_k \cos)^{-}{(k\theta)} + (b_k \sin)^{-}{(k\theta)}\right),
  \end{equation}
  is a non-negative measure on $S^1$,
  then
  \begin{equation}
    W_2^2(\mu_0,\mu_1) \leq \sum_{k\in \mathbb{Z}^+} \frac{2\pi^2}{k^2} \left(|a_k|+|b_k|\right).
  \end{equation}
  Here, $f^{-}$ stands for the negative part of the Radon measure $f$.
\end{theorem}
\begin{proof}
  We shall find at least one (a priori not optimal) transference plan from $\mu_0$ to $\mu_1$ by rearranging only $|a_k|$ or $|b_k|$ mass within an arc of length $2\pi/k$. 
  Let 
  $$
  D_k = \{(\theta, \varphi)\in S^1 \times S^1 \mid \varphi = \left(\theta + \frac{\pi}{k} \right) \mbox{ mod } 2 \pi \}, \quad k \in \mathbb{Z}^+ , 
  $$
  and $D_\infty$ be the diagonal $\{(\theta, \theta)\}$ in $S^1 \times S^1$. Consider the following coupling:
  \begin{equation}
    \kappa = \delta(D_\infty)\nu + \sum_{k\in \mathbb{Z}^+} \delta(D_{k})  \left((a_k \cos)^{+}{(k\theta)} + (b_k \sin)^{+}{(k\theta)}\right).
  \end{equation}
This coupling keeps an amount of mass in place, which is shared between $\mu_0$ and $\mu_1$, and transport the rest within one corresponding period. It follows that $\kappa$ has marginals $\mu_0$ and $\mu_1$ and is an admissible transference plan. This means that
\begin{equation}
  W_2^2(\mu_0,\mu_1) \leq \int_{S^1 \times S^1} c(\theta, \varphi)\, \dd \kappa(\theta, \varphi) = \frac{2\pi^2}{k^2} \left(|a_k|+|b_k|\right),
\end{equation}
where the cost function $c(\theta, \varphi) = \min{(|\theta - \varphi|^2, (2\pi -|\theta - \varphi|)^2 )} $ is associated with the geodesic distance along the circle.
\end{proof}

\begin{remark}
  In the proof of Theorem~\ref{thm:freq-sens}, we use a constructive approach rather than the explicit solution of the 1D OT. 
  The result holds true for high dimensional domains with boundaries. The proof needs to be modified concerning boundary treatment, and the corresponding weight is $|k|^{-2}$. It is also worth mentioning that $W_2$ is not very sensitive to oscillations and hence offers a natural weighting emphasizing the low-frequency differences. Therefore, the primary motivation for using $W_2$ is to solve large-scale errors instead of pursuing high-resolution imaging.
\end{remark}

\subsection{Gradient of quadratic Wasserstein distance}
The seismic inverse problem is that of solving for model functions in a nonlinear system. Considering the large scale of the system, the commonly used approach is to formulate the inverse problem as an optimization problem and solve it with gradient-based methods. A brief discussion of the directional differentiability properties of the quadratic Wasserstein distance along certain directions is presented here. We start by extending $\J_{\nu}(\mu) = \mathcal{T}(\mu,\, \nu)$ to a functional on $L_2$. 

Roughly speaking, the optimization is performed using linearization in a vector space and, instead of the $L_2$-norm, the total transportation cost is used as the objective function. As a result, this suggests that it is necessary to extend the functional from the probability space to the $L_2$ space. With a slight abuse of notation, we extend the functional to $\mathcal{T}:L_2 \times L_2 \rightarrow [0,+\infty]$ by
\begin{equation}
  \mathcal{T}(\mu, \nu) = 
  \left\{
    \begin{array}{>{\displaystyle}r>{\displaystyle}l}
    W_2^2(\mu, \nu),  & \quad \mbox{if } \mu, \nu \in \mathcal{P}, \\
    +\infty ,  & \quad \mbox{otherwise.}
  \end{array}
  \right.
\end{equation}
Next, we introduce a subset $\mathcal{U}\subset\mathcal{P}$, which is, in some sense, served as the ``interior'' of $\mathcal{P}$. Then, a short discussion is presented on the differentiability properties of the transportation cost $\mathcal{T}(\mu,\nu)$ over $\mathcal{U}$, see \cite{Villani2003,Santambrogio2015} for more detail and more general cases.
Discussion in this section paves the way to data encoding and minimization of the misfit between seismic data in transportation sense. 

Let $\Sigma$ be the Borel $\sigma$-algebra on the given bounded domain in $\mathbb{R}^n$ and 
\begin{equation}
\mathcal{U} = \{ \mu \in \mathcal{P} \mid \exists r>0 \mbox{ s.t. } \int \chi_{A}\, \dd \mu \geq r \int \chi_{A}\, \dd x , \quad  \forall A\in\Sigma \}.
\end{equation}

\begin{theorem}\label{thm:grad}
  Let $\mathcal{T}:L_2 \times L_2 \rightarrow [0, +\infty]$ be the extended transportation cost. Consider the functional $\gamma \mapsto \mathcal{T}(\gamma, \nu)$ for a fixed measure $\nu\in \mathcal{P}$. If $\mu\in\mathcal{U}$, then 
  \begin{equation}
  \frac{\partial \mathcal{T}(\gamma,\nu)}{\partial \gamma}(\mu) = \varphi,
  \end{equation}
  where $\varphi$ is the Kantorovich potential associated with $\mu$ and is unique up to additive constants.
\end{theorem}
\begin{proof}  
  For some fixed $\gamma \in \mathcal{U}$, consider the sequence $\{\mu_t \triangleq \mu + t(\gamma - \mu)\}$ converging to $\mu$ in the sense of 
  \begin{equation}
    \lim_{t \rightarrow 0} \mathcal{T}(\mu_t, \mu) = 0.
  \end{equation}
  By the triangle inequality on $W_2$, one gets
  \begin{equation}
    \lim_{t \rightarrow 0} \mathcal{T}(\mu_t, \nu) -\mathcal{T}(\mu, \nu) = 0, \quad \forall \nu \in \mathcal{P}.
  \end{equation}
  Let $(\varphi, \psi)$ be an optimizing pair in the Kantorovich dual formulation, i.e., 
  \begin{equation}
    \mathcal{T}(\mu, \nu) = \int \varphi \, \dd \mu + \int \psi  \, \dd \nu ,
  \end{equation}  
  and we additionally assume that $\int \varphi (x) \dd x =0$, thus making the unique determination of  $(\varphi, \psi)$.
  The sub-differentiability of $\mathcal{T}(\cdot, \nu)$ follows from the fact that $(\varphi, \psi)$ is optimal for $\mathcal{T}(\mu,\nu)$, and is not necessarily optimal for $\mathcal{T}(\mu_t,\nu)$, 
  \begin{equation}
    \begin{aligned}
         & \mathcal{T}(\mu_t,\nu) - \mathcal{T}(\mu, \nu) \\
    \geq & \left(\int \varphi \, \dd \mu_t + \int \psi  \, \dd \nu \right) - \left(\int \varphi \, \dd \mu + \int \psi  \, \dd \nu \right) \\
      =  & t \int \varphi \, \dd (\gamma-\mu).
    \end{aligned}
  \end{equation}

  For the other part of the differentiability, we denote a subsequence realizing the limit superior of $\mathcal{T}(\mu_t,\nu)$ by $\{\mu_{t_k}\}$, i.e., 
  \begin{equation}
    \lim_{k\rightarrow  +\infty} \mathcal{T}(\mu_{t_k},\nu) = \limsup_{t \rightarrow 0} \mathcal{T}(\mu_t,\nu),
  \end{equation}
  and let $(\varphi_k, \psi_k)$ be an optimizing pair for $\mathcal{T}(\mu_{t_k},\nu)$. Additionally, we assume $\int \varphi_k (x) \dd x =0$. Thus, the uniqueness of $\varphi_k$ follows by the $\dd \mu_k$-uniqueness of $\nabla \varphi_k$ and the fact that $\mu_k\in \mathcal{U}$ is positive.
  Then, we conclude from the suboptimality of $(\varphi_k, \psi_k)$ for $\mathcal{T}(\mu, \nu)$ that
  \begin{equation}
    \begin{aligned}
        &  \mathcal{T}(\mu_{t_k},\nu) - \mathcal{T}(\mu, \nu) \\
    \leq &   \left(\int \varphi_k \, \dd \mu_{t_k} + \int \psi_k  \, \dd \nu \right) - \left(\int \varphi_k \, \dd \mu + \int \psi_k  \, \dd \nu \right)  \\
    = &  t_k \int  \varphi_k \, \dd (\gamma-\mu).
    \end{aligned}
  \end{equation}
From the stability of the optimal transference mapping and Brenier's theorem \cite{Villani2003}, we know $\varphi_k \rightharpoonup \varphi$. Hence 
  \begin{equation}
    \lim_{t\rightarrow 0} \frac{\mathcal{T}(\mu_t,\nu) - \mathcal{T}(\mu, \nu)}{t}  = \int \varphi \, \dd (\gamma-\mu).    
  \end{equation}
The uniqueness of $\varphi$ up to additive constants follows by noting that $\nabla \varphi$ is $\dd \mu$-unique and $\mu$ is positive everywhere.

\end{proof}

\begin{remark}[On the strictly positive range of the encoding mapping]
Usually functions differing on a measure-null set only are not distinguished. In the inverse problem context, one compares two encoded data and does not expect them to be invisible to each other. The definition of $\mathcal{U}$ originates from the idea that any two elements of $\mathcal{U}$ should be absolutely continuous to each other and the observation $\mathcal{U}+\varepsilon \xi \subset \mathcal{U}$ for small $\varepsilon$ and bounded mean-zero perturbation $\xi$. On the other hand, the positiveness of $\mu$ is required to ensure that the derivative $\varphi$ is unique up to additive constants over the whole domain. In the next section, this uniqueness will be used to show the associated gradient for the velocity model is unique, and therefore, the adjoint state method is well defined. Last but not least, the uniform lower bound $r$ in the definition of $\mathcal{U}$ is to give more space for the line search in the optimization. 
\end{remark}

\section{Encoding methods}
\label{sec:encode}
In this section, we investigate the criterion for selecting a proper encoding method to transfer the non-Wasserstein-measurable seismic data into PDFs. A simple but quite useful strategy using the softplus function is presented, and some useful properties are examined.
Our goal is to make the data misfit measurable using the Wasserstein distance and efficiently calculate the associated gradient. In this perspective, we suggest the following strategies to choose encoding map $\mathcal{D}$:
\begin{enumerate}
  \item The range of $\mathcal{D}$ is contained in $\mathcal{U}$;
  \item $\mathcal{D}$ is differentiable and invertible;
  \item $\mathcal{D}$ is a pointwise mapping, i.e., $(\mathcal{D}\circ u)(x) =  \mathcal{D}(u(x))$.
\end{enumerate}
The first point guarantees the existence and uniqueness (up to an additive constant) of the first variation of the transportation cost. The second one makes the mapping compatible with Quasi-Newton type methods.
The third point is purely for the sake of efficiency. 
Usually, to match the mass of the encoded data, a normalization procedure is involved, and it is hard to ensure the invertibility of the encoding map. A common solution for this issue is to keep the total mass aside and use it when need to invert the encoding map. Therefore, only the mass-distribution will be used to calculate the data misfit, which is consistent with the consensus that the seismic inversion depends primarily on phase, not amplitude information.
For example, one can map $u$ to $(\tilde{u}/\langle \tilde{u} \rangle, \langle \tilde{u} \rangle)$ with $\tilde{u} = \log(1+\exp(u))$, and use the first element only for the misfit calculation; the second element is needed when inverting the map.
In the following sections, encoding mappings that meet the above three conditions will be referred to as regular mappings.

\subsection{Uniqueness of the gradient in the adjoint state method}
According to Theorem~\ref{thm:grad}, for any $\mu\in \mathcal{U}$, the first variation of the transportation cost exists and is unique almost everywhere up to additive constants. 
Apparently, for all regular encoding maps, one expects that the gradient ${\dd \mathcal{J}}/{\dd m}$ in the adjoint state method does not depend on the particular choice of the Kantorovich potential $\varphi$. The following theorem presents a rigorous proof of this result.

\begin{theorem}
  Let $m$ be a parameter model and $u$ the data associated with $m$ as in \eqref{eq:wave}. For any fixed $\nu\in \mathcal{U}$, the value of 
\begin{equation}
  \frac{\dd }{\dd m} \mathcal{T}(\mathcal{D}(u(m)),\nu)
\end{equation}
does not depend on the particular law by which the Kantorovich potential $\varphi$ is chosen, provided that  $\mathcal{D}$ is differentiable.
\end{theorem}
\begin{proof}
  Let $m_0,m_1$ be the first variations of $\mathcal{T}$ obtained with the particular choice of the Kantorovich potential, say $\varphi_0$ and $\varphi_1$, respectively. Recall from the adjoint state method \eqref{eq:W2-adj-stat} and \eqref{eq:adjoint-eq-W2}, that $m_k$'s are of the form
  \begin{equation}\label{img-cond}
    m_k = \int_0^T \, v_k \partial_t^2 u \dd t, \quad k=1,2,
  \end{equation}
  where $u$ is the background wavefield, and $v_k$ solves the adjoint wave equation with $\mathcal{D}'[u]^* (\varphi_k)$ as the right-hand side:
  \begin{equation}\label{adj-wave}
    \left\{
    \begin{array}{>{\displaystyle}r>{\displaystyle}l}
    \left( m(x) \frac{\partial_2}{\partial t^2} - \Delta \right) v_k(x,t) &  =  \mathcal{D}'[u]^* (\varphi_k) , \\
    v_k(x,T)  & =  0, \\
     \partial_t v_k(x,T) & =  0 .
     \end{array}
    \right.
  \end{equation}
  By Theorem~\ref{thm:grad}, we find that 
  \begin{equation}\label{cst-diff}
    \varphi_0 -\varphi_1 = c
  \end{equation}
  for some constant $c$. We claim that
  \begin{equation}
    \int_\Omega (m_0-m_1) h_m \dd x = 0, \quad \forall h_m \in L_2.
  \end{equation}
  To prove this, we consider an auxiliary wavefield $h_u$ that solves the wave equation with $-h_m \partial_t^2 u$ as the right-hand side:
  \begin{equation}\label{linear-wave}
    \left\{
    \begin{array}{>{\displaystyle}r>{\displaystyle}l}
    \left( m(x) \frac{\partial_2}{\partial t^2} - \Delta \right) h_u(x,t) &  =  -h_m \partial_t^2 u , \\
    h_u(x,0)  & =  0, \\
     \partial_t h_u(x,0) & =  0 .
     \end{array}
    \right.
  \end{equation}
  Then, it follows that
  \begin{equation}
    \begin{aligned}
      & \int_\Omega (m_0-m_1) h_m \dd x \\
      = & \int_{\mathbb{R}^n} \left(\int_0^T (v_0-v_1) \partial_t^2 u  \, \dd t \right) h_m \dd x \\
      = & -\int_{\mathbb{R}^n} \int_0^T \left(v_0-v_1\right) \left( m(x) \frac{\partial_2}{\partial t^2} - \Delta \right) h_u(x,t)  \, \dd t \, \dd x \\
      = & -\int_{\mathbb{R}^n} \int_0^T h_u \left(\mathcal{D}^{'}[u]^*\varphi_0 - \mathcal{D}^{'}[u]^* \varphi_1 \right)  \, \dd t \, \dd x \\
      = & -\int_{\mathbb{R}^n} \int_0^T c \mathcal{D}'[u] (h_u) \, \dd t \, \dd x \\
      = & 0.
    \end{aligned}
  \end{equation}
  In the above derivation, the first equality is from \eqref{img-cond}; substituting for $h_m \partial_t^2 u$ using \eqref{linear-wave}, we obtain the second equality; the third equality employs \eqref{adj-wave} and integration by parts twice; then, we use the definition of the adjoint operator and \eqref{cst-diff} to conclude the proof.
\end{proof}

\subsection{Encoding with softplus function}
We now turn to the formulation of an encoding map using the softplus function.
The Logistic function is defined as
\[
 f(x) = \frac{L}{1+e^{-\beta(x-x_0)}},
\]
where $x_0$ is the value of the sigmoid's midpoint, $L$ is the curve's maximum value, and $\beta$ is the steepness of the curve. The standard logistic function is the one with parameters $(\beta=1 , x_0 = 0 , L=1)$, which yields
\[
 f(x) = \frac{e^x}{e^x + 1} = \frac{1}{1+ e^{-x}}.
\]
The logistic function is useful since it can take any real number, whereas the output always takes values between zero and one and hence is interpretable as a PDF. In practice, due to the nature of the exponential function $e^{-x}$, it is often sufficient to compute the standard logistic function for $x$ over a small range of real numbers, such as a range contained in $[-5, 5]$.
The anti-derivative of the logistic function,
\[
 f(x) = \log{(1+e^x)}
\]
is widely used in logistic regression, which is used in various areas, including machine learning and social sciences. The output also takes a positive value. Its derivative shows that the variance for negative input value is small. The graph of the function (Figure~\ref{fig:ins-func}) shows that the behavior of $f(x)$ is flat when $x<0$ and is very similar to $f(x) = x$ when $x>0$.
\begin{figure}[hbtp]
  \begin{center}
  \includegraphics[width=.6\textwidth]{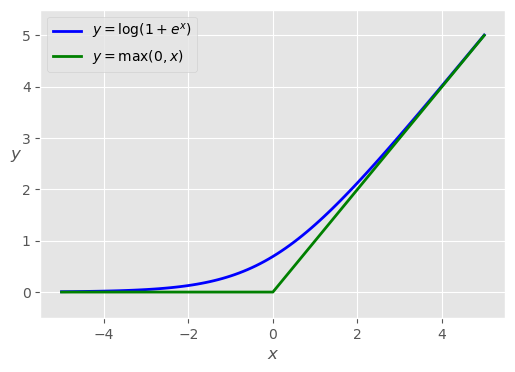}
  \end{center}
  \caption{Softplus function and projection to the positive part}
  \label{fig:ins-func}
\end{figure}

We use the following operation
\begin{equation}
\label{eq:encode-func}
\tilde{u}(t) = \frac{1}{|\beta|} \log (1+ e^{\beta u(t)})
\end{equation}
composited with the normalization 
$$
\tilde{u} \mapsto \frac{\tilde{u}}{\langle \tilde{u} \rangle} 
$$
to encode the seismic data into PDFs, where $\langle \cdot \rangle$ denotes the averaging operation.
It is easy to check that
\begin{equation}\label{asym_bahav}
\lim_{\beta \rightarrow +\infty}\tilde{u} = u^+ \triangleq \max(u,0)  \mbox{ and }
\lim_{\beta \rightarrow -\infty}\tilde{u} = u^- \triangleq \max(-u,0) ,
\end{equation}
and the convergence is uniform. The above asymptotic behavior is an important advantage of this encoding procedure. 
One can expect the Wasserstein distance of the functions processed using this differentiable encoding method to show similar behavior as the one using $u^{+}$ while the smoothness preserved.
According to the stability of the optimal transference plans \cite[Corollary 5.23]{Villani2008}, we can identify the convex functional on the seismic data $u$ by checking its convexity on $u^+$ and $u^-$. In practice, large $\beta$ can be chosen for better convexity in the objective function, but care should be taken to avoid the gradient-vanishing problem and overflow errors. 

\subsection{Convexity of the encoded data}
We conclude this section by examining the convexity under different measurement methods.
The main motivation for using the OT metric in the seismic inversion is to exploit its convexity to the translation and dilation, which are the primary data mismatch types.
In \cite{Engquist2016}, it is proved that the quadratic Wasserstein distance is convex with respect to translation and dilation, even in the case of a mixture of the two. In general, this convexity cannot be preserved after encoding. 
Roughly speaking, the encoding map can be interpreted as a procedure to generate non-negative functions from seismic data via adding/removing mass pointwise. After encoding, the endpoint, $t=0,T$, can be a source or sink of mass. Hence the transportation cost is no longer convex to the translation and dilation. 

Using the properties in \eqref{asym_bahav} and \cite[Theorem~2.1-2.3]{Engquist2016}, one can easily show that the encoded data using softplus function bears the asymptotic convexity when the pre-encoding data has compact support. 
Figures~{\ref{fig:geodesic_L2}-\ref{fig:geodesic_Logi}} present the interpolations of a Ricker wavelet $p_0(t)$ and its translation $p_1(t)=p_0(t-0.6)$ in $L_2$, $W_2$ with adding-constant encoding method, and $W_2$ with softplus encoding, respectively. 
Unsurprisingly, the $L_2$ one calculates the interpolation in a pointwise manner; the encoding method using added constants shows a phenomenon of local transportation; by contrast, the one using softplus function accurately captures the translation information.

\begin{figure}[hbtp]
  \begin{center}  
    \includegraphics[width=.75\textwidth]{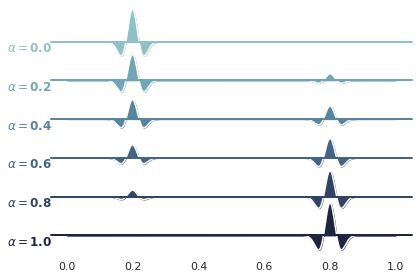}
  \end{center}
  \caption{$\displaystyle{\argmin_p \left((1-\alpha)\|p-p_0\|_2^2 + \alpha \|p-p_1\|_2^2\right)}$}
  \label{fig:geodesic_L2}
\end{figure}

\begin{figure}[hbtp]
  \begin{center}  
    \includegraphics[width=.75\textwidth]{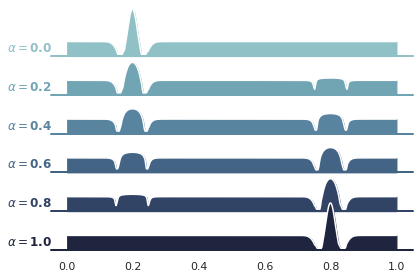}  
  \end{center}
  \caption{$\displaystyle{\argmin_p \left((1-\alpha)\mathcal{T}\left(p,\frac{p_0+c}{\langle p_0+c\rangle}\right) + \alpha \mathcal{T}\left(p,\frac{p_1+c}{\langle p_1+c\rangle}\right)\right)}$}
  \label{fig:geodesic_ushft}
\end{figure}

\begin{figure}[hbtp]
  \begin{center}  
    \includegraphics[width=.75\textwidth]{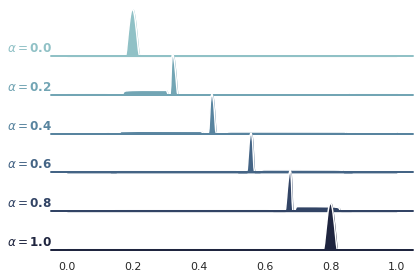}
  \end{center}
  \caption{$\displaystyle{\argmin_p \left((1-\alpha)\mathcal{T}(p,\mathcal{D}(p_0)) + \alpha \mathcal{T}(p,\mathcal{D}(p_1))\right),\quad \mathcal{D}(p) = \frac{\log(1+e^p)}{\langle\log(1+e^p)\rangle}}$}
  \label{fig:geodesic_Logi}
\end{figure}

\section{Numerical examples}\label{sec:num}
In this section, the properties of our proposed algorithm are illustrated through two numerical experiments. We first use simple structural models to investigate the relationship between the convexity of the misfit function and the encoding parameter $\beta$. The numerical experiment indicates that one can tune $\beta$ to alleviate the local minima problem. 
Then, an inversion is performed on the 2D benchmark Marmousi model \cite{Versteeg1994} to demonstrate the effectiveness of our method.
To take advantage of the 1D explicit solution and avoid confusion on transportation over different units, a trace-by-trace strategy is employed to compute the objective function and the adjoint source. That is, we use the objective function
\begin{equation*}
    \J(u(x,t),d(x,t)) = \int \mathcal{T}(u(x,t),d(x,t)) \, \dd x.
\end{equation*}

\subsection{The landscape of objective functions}
We start our study of numerical experiments with a numerical investigation of the landscape of the misfit function. The experiment is performed on a family of 2D models with two feature variables. The receivers are uniformly distributed at an interval of $40$ m over the top surface with $16.85$ km length, and a point source is located in the middle of the receivers. We use the following formula to build the velocity models:
\begin{equation}
\label{eq:linear-vel}
v(x,z) =
\left\{
\begin{array}{>{\displaystyle}r>{\displaystyle}l}
1500, \quad \mbox{when } z<50,
\\
v_0 + \alpha z,  \quad \mbox{when } z \geq 50.
\end{array}
\right.
\end{equation}
A band-pass filter at $3-18$ Hz is applied to the source function and the data to imitate the actual exploration seismic data.
The reference data is obtained with velocity model constructed with $v_0 = 2000$ m/s and $\alpha = 0.7 \, \mbox{s}^{-1}$.
Figure~\ref{fig:misfit-convexity} shows the misfit functions as functions of $v_0$ and $\alpha$. For better comparison, we normalize the misfit using its maximum value.

The landscape using $L_2$ metric is shown in Figure~\ref{fig:mf-L2}. Due to the high nonlinearity of the inverse problem and limited acquisition geometry, there are many local minima even for this simple-structured model. The gradient of the misfit function gives no information or even wrong information once the background velocity is too far from the reference one. To arrive at the global minimum using a gradient-based descent method, one needs to start from an initial model within the same basin as the global minimum.

We further investigate the applicability of convexifying the $W_2$ misfit using encoding parameter $\beta$. This experimental setting provides a perfect scenario for the quadratic Wasserstein metric, since the number of the seismic events stay the same.
Actually, it is easy to prove the asymptotic convexity of the objective function rigorously. Therefore, our goal is to eliminate the local minima by tuning $\beta$. Figure~\ref{fig:mf-W2-1}-\ref{fig:mf-W2-3} displays the landscapes with gradually increasing $\beta$. It demonstrates that the larger $\beta$, the less local minima.
It is also interesting to note that $W_2$ misfit functions are smoother than the $L_2$ one, which is associated with the fact that the regularity of the optimal transportation map is one degree higher than that of the seismic data.

\begin{figure}[ht]
\centering
\subfigure[$L_2$]{
    \includegraphics[width=.5\textwidth,trim={2cm 0.2cm 1cm 1.2cm},clip]{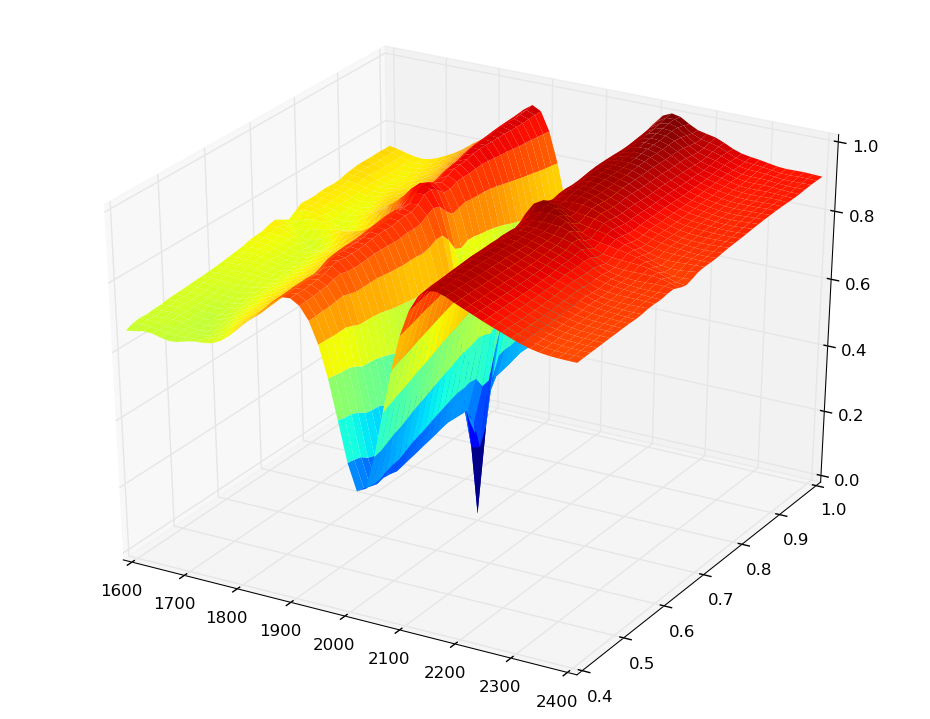}
    \label{fig:mf-L2}
}%
\subfigure[$W_2$ with $\beta=0.8$]{
    \includegraphics[width=.5\textwidth,trim={2cm 0.2cm 1cm 1.2cm},clip]{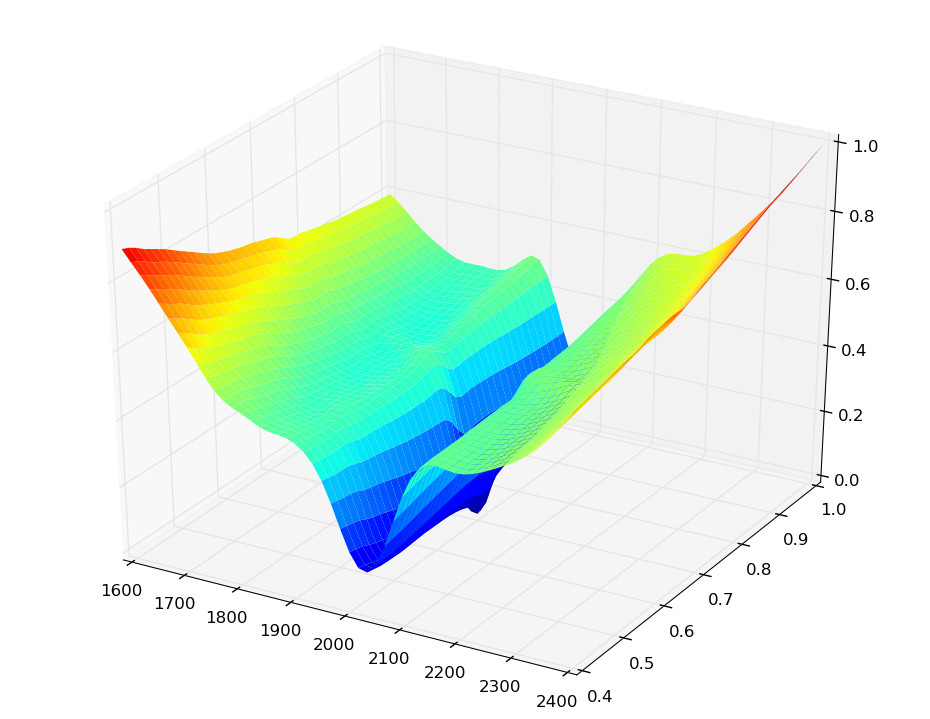}
    \label{fig:mf-W2-1}
} \\
\subfigure[$W_2$ with $\beta=1.2$]{
    \includegraphics[width=.5\textwidth,trim={2cm 0.2cm 1cm 1.2cm},clip]{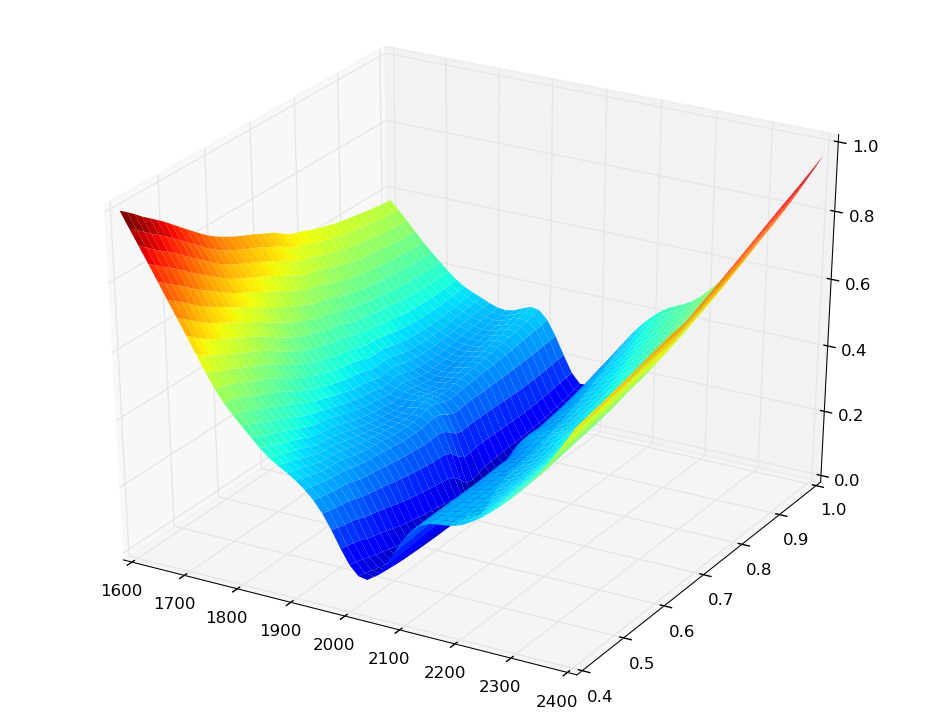}
    \label{fig:mf-W2-2}
}%
\subfigure[$W_2$ with $\beta=6$]{
    \includegraphics[width=.5\textwidth,trim={2cm 0.2cm 1cm 1.2cm},clip]{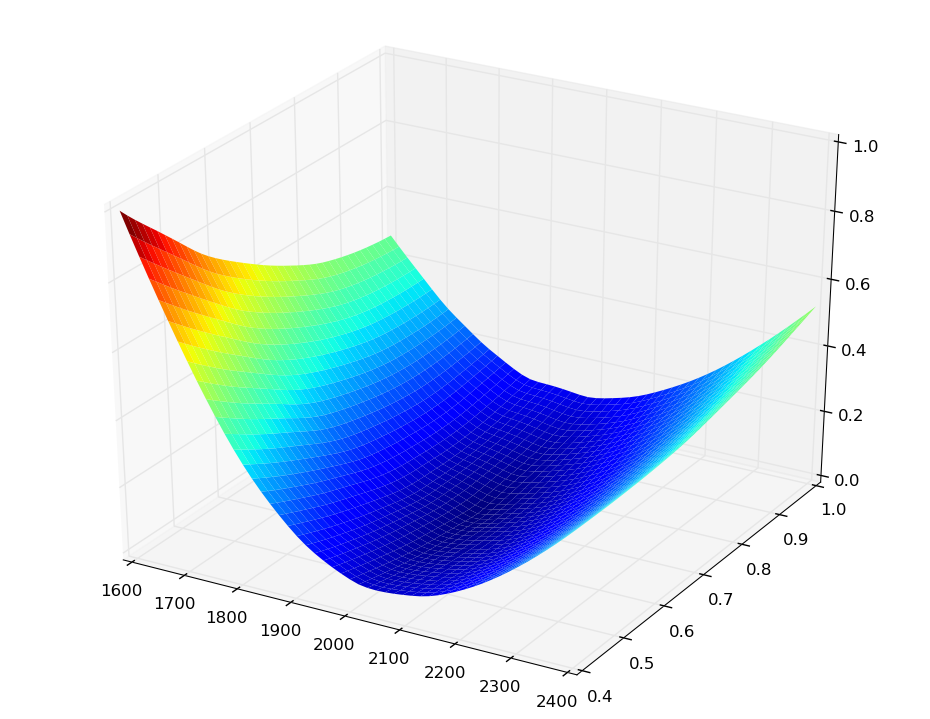}
    \label{fig:mf-W2-3}
}
\caption[]{Comparison of landscapes}
\label{fig:misfit-convexity}
\end{figure}


\subsection{Inversion on the Marmousi model}

In the following experiments, we use the Marmousi benchmark model \cite{Versteeg1994}. The true velocity model is shown in Figure~\ref{fig:marm-true}. A $921 \times 319$ grid is used to represent an approximately 9.2 km $\times$ 3.2 km area. In both $L_2$ and $W_2$ cases, a heavily smoothed model from the true one, as shown in Figure~\ref{fig:marm-init}, is used as the initial model for the iterative gradient-based descent method.

In this experiment, a perfectly matched layer (PML) absorbing boundary condition is applied to the domain boundaries except for the top free surface. The synthetic data is generated with an array of equally spaced 201 sources at depth 8 m and 461 receivers at depth 12 m distributed over the model's top surface.  
The source signature is the Ricker wavelet with a center frequency of 10 Hz, and the recording time is $4.5$ s. A 3-18 Hz band-pass filter is applied to the source and the data to imitate the actual seismic data in geophysical exploration.
For the modeling and inversion, we use Devito \cite{Lange2016} to solve the acoustic wave equation and the associated adjoint state equation. The numerical solution is obtained with a finite-difference scheme, which is forth-order accurate in space and second-order in time. We employ the limited-memory BFGS method with box constraints \cite{Byrd1995} implemented in SciPy \cite{Virtanen2020} for the optimization. All the numerical experiments stop when the decrease of the objective function meets the stopping criteria,
\begin{equation*}
    \frac{\J_k - \J_{k+1}}{\max (\J_k,\J_{k+1}, 1)} < 10^{-5}.
\end{equation*}
The inversions using $L_2$ and $W_2$ with softplus encoding stop after $20$ and $27$ iterations, respectively. 
The reconstruction results are displayed in Figure~\ref{fig:marm-L2} and \ref{fig:marm-W2}.
Due to the significant difference between the initial model and the true model, the least-squares formulation suffers from a cycle-skipping issue. It is clear that the inversion using $L_2$ metric terminates with an incorrect velocity model.
\begin{figure}[ht]
\centering
\subfigure[True model]{
\includegraphics[width=.5\textwidth]{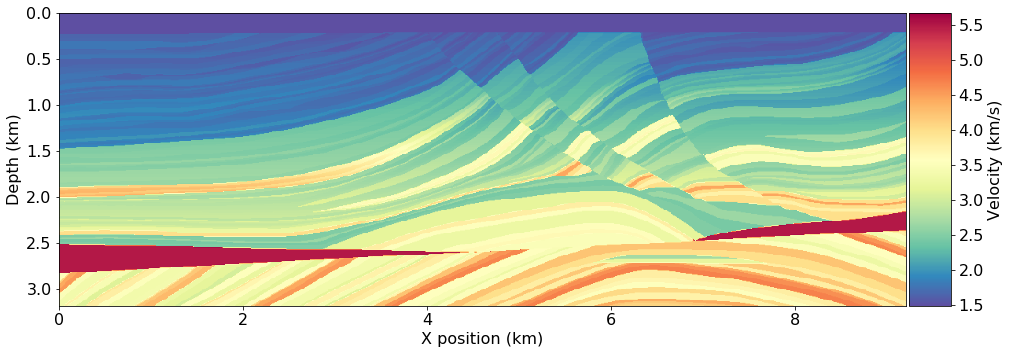}
\label{fig:marm-true}
}%
\subfigure[Initial model]{
\includegraphics[width=.5\textwidth]{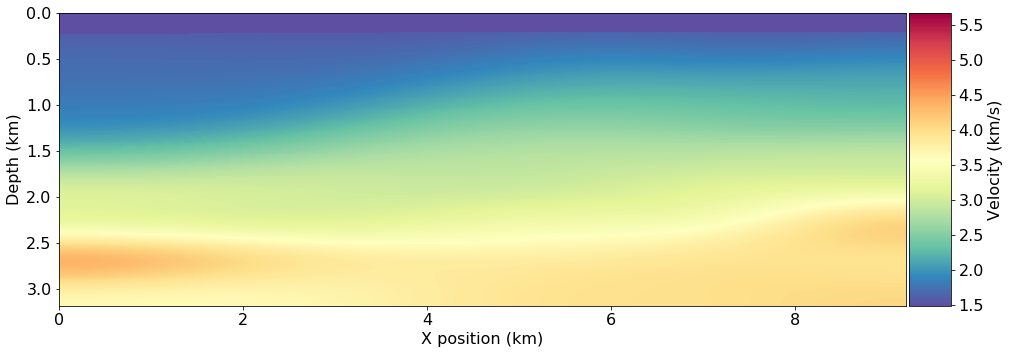}
\label{fig:marm-init}
} \\
\subfigure[Reconstructed model with $L_2$]{
\includegraphics[width=.5\textwidth]{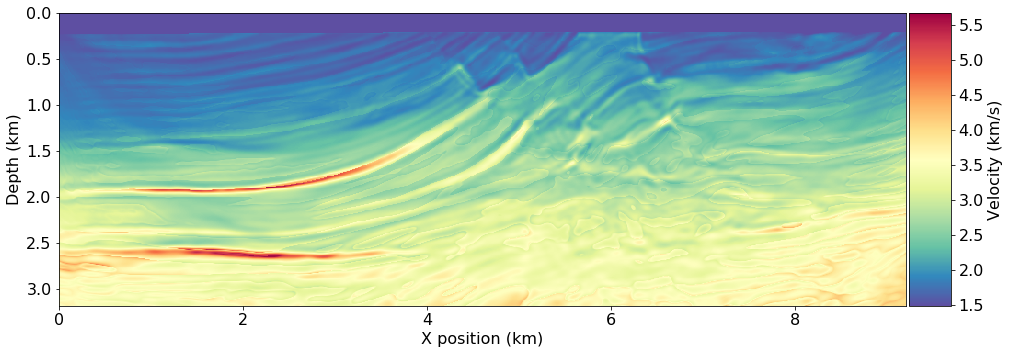}
\label{fig:marm-L2}
}%
\subfigure[Reconstructed model with $W_2$]{
\includegraphics[width=.5\textwidth]{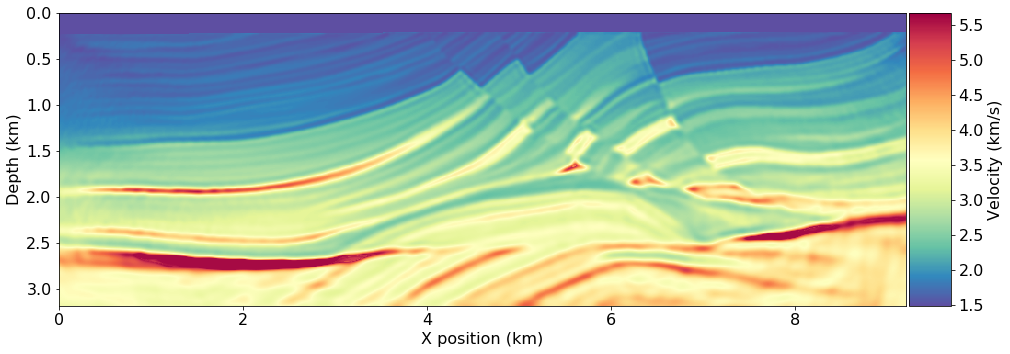}
\label{fig:marm-W2}
}
\caption{Velocity models in the numerical experiment with Marmousi model}
\label{fig:Marmousi-inv}
\end{figure}

\begin{figure}[tbhp]  
  \centering
  \subfigure[Objective function]{
  \includegraphics[width=.5\textwidth]{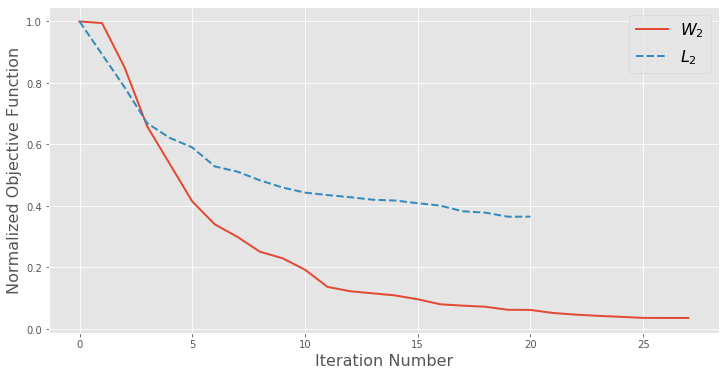}
  \label{fig:data-misfit}
  }%
  \subfigure[Reconstruction error]{
    \includegraphics[width=.5\textwidth]{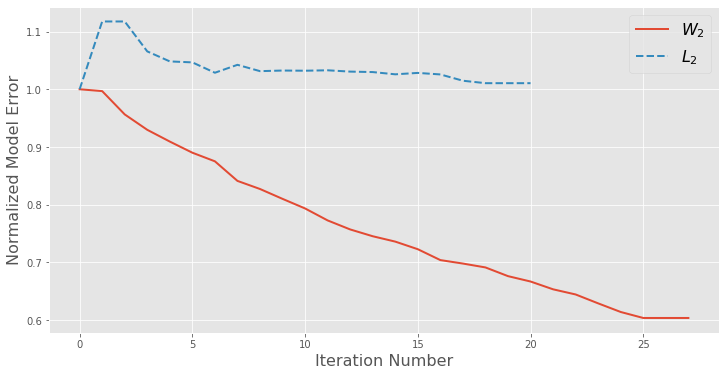}    
    \label{fig:slowness-misfit}
  }\\
  \caption{Data and model error}
  \label{fig:data-model-error}
\end{figure}

We present a vertical slice at $x=3$ km in Figure~\ref{fig:vel-slices}. By contrast, the $L_2$ metric produces low-velocity artifacts, which is strong evidence of cycle-skipping. Hot spots of slowness errors are shown in Figure~\ref{fig:hot-spot-ini}-\ref{fig:hot-spot-W2}. The $W_2$ metric correctly reconstructs the area swept by the diving waves. Some finer structures in the deeper region, mainly reflectors, can be improved using further iterations with a $L_2$ metric.
The analysis in Theorem~\ref{thm:freq-sens} suggests that $L_2$ should be better for the inversion of details when it does not suffer from the cycle-skipping issue anymore. We use a fixed encoding parameter $\beta=2.0$, as we can switch to $L_2$ metric once the cycle-skipping problem is overcome.

\begin{figure}[tbhp]  
  \centering
  \subfigure[Vertical slices]{
  \includegraphics[width=.5\textwidth]{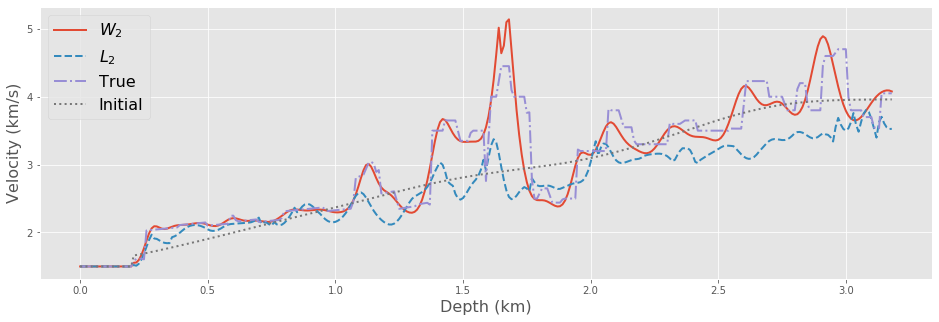}  
  \label{fig:vel-slices}
  }%
  \subfigure[Relative difference - Initial model]{
    \includegraphics[width=.5\textwidth]{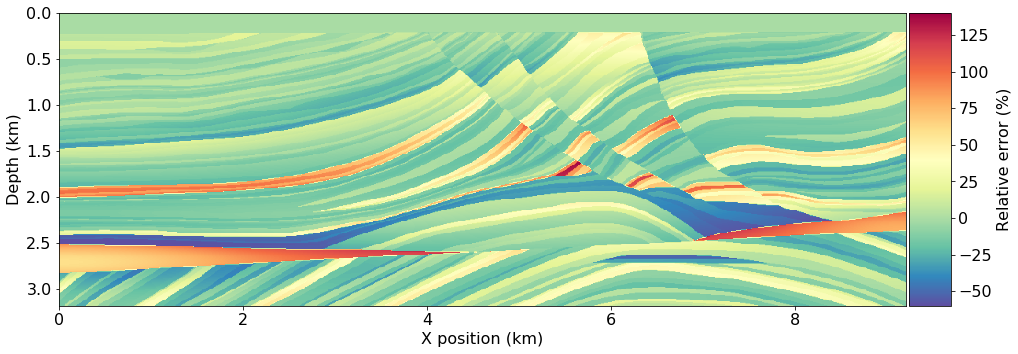}    
    \label{fig:hot-spot-ini}
  }\\
  \subfigure[Relative difference - $L_2$]{
  \includegraphics[width=.5\textwidth]{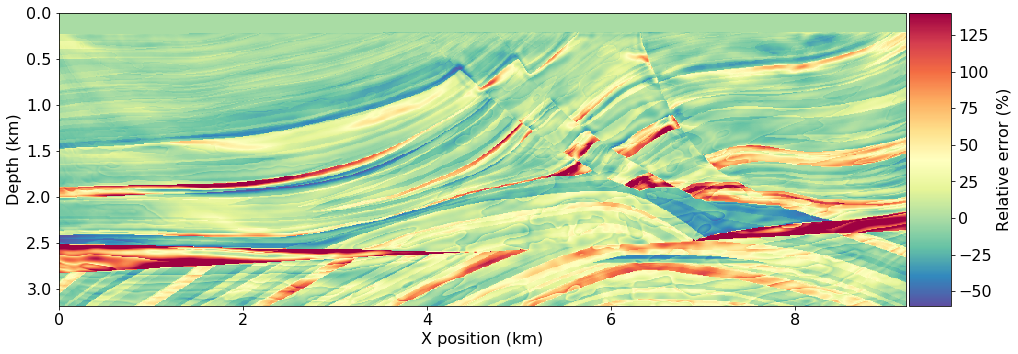} 
  \label{fig:hot-spot-L2}
  }%
  \subfigure[Relative difference - $W_2$]{
    \includegraphics[width=.5\textwidth]{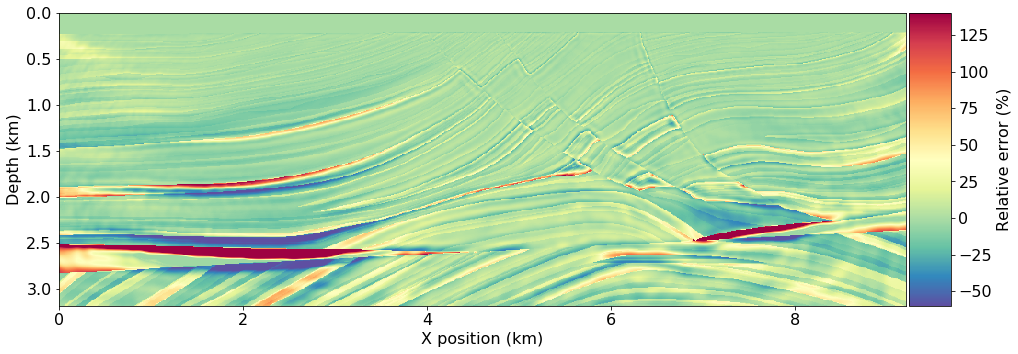}    
    \label{fig:hot-spot-W2}
  }\\
  \caption{Vertical slices and relative slowness difference $(m - m_{true})/m_{true}$}
  \label{fig:slice-hot-spot}
\end{figure}




\section{Conclusion}
We investigated the properties of the objective function for FWI using the quadratic Wasserstein metric and proper encoding methods. We rigorously prove that the quantity ${\dd \J}/{ \dd m}$, obtained using the adjoint state method, does not depend on the particular choice of the Kantorovich potential if one chooses the encoding method properly. In particular, transportation metric with softplus encoding has asymptotic convexity concerning time-shift and dilation.
It helps one extract time-shift information more accurately, thus provides the velocity model with appropriate large-scale changes, and mitigates the cycle-skipping problem. 

Another point that should be stressed is that, based on the result in Theorem~\ref{thm:freq-sens}, the transportation type objective function enhances low-frequency information as $\dot{H}^{-1}$ does. Thus, $W_2$ is more appropriate to be used when the initial model is far from the true model. Once the cycle-skipping issue is fixed, it is better to switch to a $L_2$ metric for fast high-resolution reconstruction.

On two numerical examples, we show the feasibility of the proposed method.
The first one using two-parameter models to illustrate how the softplus encoding parameter recasts the landscape of the objective function. In the second example, we demonstrate the accuracy and efficiency of our method when applied to synthetic data generated by the Marmousi model. We realize that a subtle choice of encoding parameter is not required. Typically, a value between $1$ and $5$ will fit most cases.

As we mainly focus on applying the transportation metric on seismic data, optimization techniques only involving first-order derivatives are adopted. The Wasserstein metric with softplus encoding can be extended to be suitable for Newton's method or other second-order algorithms. Moreover, other than treating $\mathcal{T}(\cdot, \nu)$ as a function defined on $L_2$ and considering only the differential formulation in Euclidean sense, another natural strategy is to use Otto's calculus \cite[Chapter 15]{Villani2008} and consider optimization using gradient flows in the Wasserstein space. These approaches will be investigated in a future article.

\appendix
\setcounter{section}{1}

\begin{algorithm}[tbh]
\caption{Calculation of $\mathcal{T}(p_0,p_1)$ }\label{alg:W2}
{\fontsize{12}{24}\selectfont
\begin{algorithmic}[1]
\Procedure{Pseudo-inverse}{$f_0,f_1,t$}\Comment{Calculation of $f_0^{-1}(f_1(t))$}
\State $m \gets 1$
\For{$k \gets 1 \, , N$}
  \State $\tilde{f} = f_1(t_k)$
  \While{$f_0(t_m) < \tilde{f}$}
  \State $m \gets m+1$
  \EndWhile
  \If{$m=1$}
  \State $\varphi(t_k) \gets t_m$
  \ElsIf{$m=N$ and $f(t_m)<\tilde{f}$}
  \State $\varphi(t_k) \gets t_m$
  \Else
  \State $\alpha \gets \frac{\tilde{f} - f(t_{m-1})}{f(t_m) - f(t_{m-1})}$
  \State  $\varphi(t_k) \gets (1-\alpha)t_{m-1} + \alpha t_m$
  \EndIf
\EndFor
\State \textbf{return} $\varphi$\Comment{The function $\varphi$ is $f_0^{-1}(f_1(t))$}
\EndProcedure
\Procedure{$\mathcal{T}$}{$p_0,p_1,t$}\Comment{Calculation of $\mathcal{T}(p_0,p_1)$}
\For{$k \gets 1 \, , N$}
  \State $f_0(t_k) \gets f_0(t_{k-1}) + p_0(t_k)$
  \State $f_1(t_k) \gets f_1(t_{k-1}) + p_1(t_k)$
\EndFor
\State $\varphi \gets$ Pseudo-inverse$(f_0,f_1,t)$ \Comment{The function $\varphi$ is $f_0^{-1}(f_1(t))$}
\State w$\gets 0$
\For{$k \gets 1 \, , N$}
  \State $w \gets w + p_1(t_k) \,(\varphi(t_k) - t_k)^2 $
\EndFor
\State \textbf{return} $w$\Comment{The value of $\mathcal{T}(p_0,p_1)$ is $w$}
\EndProcedure
\end{algorithmic}
}
\end{algorithm}

\begin{algorithm}[tbh]
\caption{Calculation of $\frac{\partial}{\partial p_1}\mathcal{T}(p_0,p_1)$ }\label{alg:Frec-W2}
{\fontsize{12}{24}\selectfont
\begin{algorithmic}[1]
\Procedure{Integration-helper}{$(u, \varphi, t)$} \Comment{Calculation of $\int_u^1 (t-\varphi) \, \dd t $}
\State $m \gets n-1$
\State $S \gets 0 $
\For{$k \gets N\, , 1$}
  \State $\tilde{f} = u(t_k)$
  \While{$t_m > \tilde{f}$ and $m > 1$}
    \State $S \gets S + \frac{1}{2} \, ((t_m - \varphi(t_m)) + (t_{m+1} - \varphi(t_{m+1}))) (t_{m+1} - t_m)$
    \State $m \gets m-1$
  \EndWhile
  \If{$m=0$}
    \State $\xi(t_k) \gets S$
  \Else
    \State $\alpha \gets \frac{\tilde{f} - t_m}{t_{m+1} - t_m}$
    \State $\xi(t_k) \gets S + \frac{1}{2} ((1-\alpha)(t_m - \varphi(t_m)) + (1+\alpha)(t_{m+1} - \varphi(t_{m+1}))) \, (t_{m+1} - \tilde{f})$
  \EndIf
\EndFor
\State \textbf{return} $\xi$\Comment{$\xi(s) = \int_{u(s)}^1 (t-\varphi) \, \dd t $}
\EndProcedure

\Procedure{gradient}{$p_0,p_1,t$}\Comment{Calculation of $\frac{\partial}{\partial p_1}\mathcal{T}(p_0,p_1)$}
\For{$k \gets 1 \, , N$}
  \State $f_0(t_k) \gets f_0(t_{k-1}) + p_0(t_k)$
  \State $f_1(t_k) \gets f_1(t_{k-1}) + p_1(t_k)$
\EndFor
\State $\varphi_0 \gets$ Pseudo-inverse$(f_0,f_1,t)$ \Comment{$\varphi_0$ equals $f_0^{-1}(f_1(t))$}
\State $\varphi_1 \gets$ Pseudo-inverse$(f_1,f_0,t)$ \Comment{$\varphi_1$ equals $f_1^{-1}(f_0(t))$}
\State $\xi \gets $ Integration-helper$(\varphi_0, \varphi_1,t)$ \Comment{$\xi(t)$ equals $\int_{f_0^{-1}(f_1(t))}^1 \left(s- f_1^{-1}(f_0(s))\right) \, \dd s$}
\For{$k \gets 1 \, , N$}
  \State $\zeta(t_k) \gets (\varphi_0(t_k) - t_k )^2 + 2 \xi(t_k)$
\EndFor
\State \textbf{return} $\zeta$\Comment{$\zeta$ equals $\frac{\partial}{\partial p_1}\mathcal{T}(p_0,p_1)$}
\EndProcedure
\end{algorithmic}
}
\end{algorithm}

\bibliographystyle{acm}
\bibliography{Ref}
\end{document}